\documentclass[leqno]{amsart}

\usepackage[english]{babel}
\usepackage{amssymb}
\usepackage{latexsym}
\usepackage{amsfonts}
\usepackage{amsmath}
\usepackage{amssymb}
\usepackage{mathtools}
\usepackage{amsthm}
\usepackage{amsxtra}
\usepackage{dsfont}
\usepackage{amscd}
\usepackage{graphicx,subfigure}
\usepackage{srcltx}
\usepackage{esint}
\usepackage{pstricks,pst-coil,pst-plot,pst-3dplot,pst-math}

\DeclareMathOperator{\ind}{\mathrm{ind}} 
\newcommand{\R}{\mathbb{R}}
\newcommand{\N}{\mathbb{N}} 
\newcommand{\Z}{\mathbb{Z}}

\newcommand{\sign}{\mathop\mathrm{sign}\nolimits}

\newcommand{\D}{\partial}
\newcommand{\X}{\times}

\newcommand{\cl}[1]{\overline{#1}}
\newcommand{\caratt}[1]{\mathds{1}_{#1}}

\newtheorem{definition}{Definition}[section]
\newtheorem{theorem}{Theorem}[section]
\newtheorem{corollary}[theorem]{Corollary}
\newtheorem{lemma}[theorem]{Lemma}
\newtheorem{example}[theorem]{Example}
\newtheorem{proposition}[theorem]{Proposition}
\newtheorem{remark}[theorem]{Remark}

\renewcommand{\emph}[1]{\textbf{#1}}

\numberwithin{equation}{section}

\date{\today}

\begin{document}

\title[On a class of coupled differential equations on manifolds]{On
  the set of harmonic solutions of a class of perturbed coupled and
  nonautonomous differential equations on manifolds} \author[L.\
Bisconti]{Luca Bisconti} \email[L.\ Bisconti]{luca.bisconti@unifi.it}
\author[M.\ Spadini]{Marco Spadini} \email[M.\
Spadini]{marco.spadini@unifi.it} 

\begin{abstract} We study the set of $T$-periodic solutions of a class
  of $T$-periodically perturbed coupled and nonautonomous differential
  equations on manifolds. By using degree-theoretic methods we obtain
  a global continuation result for the $T$-periodic solutions of the
  considered equations.
\end{abstract}

\maketitle

\section{Introduction} 
This paper is concerned with topological properties of the set of harmonic solutions of 
a class of parametrized coupled periodic differential equations. 
Namely, we consider a perturbed  periodic linear nonhomogeneous differential equation in $\R^k$
coupled with a (tangent) periodic perturbation of the zero vector field on a smooth differentiable 
manifold. More precisely, let $M \subseteq \R^s$ be a boundaryless smooth manifold, given $T>0$, 
we consider $T$-periodic solutions to the following system of equations on $\R^k\X M$
\begin{equation} \label{eq:main} 
\left\{ \begin{array}{l} \dot x =
      A(t)x + c(t) + \lambda f_1(t, x, y,\lambda), \\
      \dot y = \lambda f_2(t, x, y,\lambda),
    \end{array} \right. \quad
  \lambda\geq 0,
\end{equation}
where $A\colon\R\to GL(\R^k)\subseteq\R^{k\X k}$, is a continuous matrix-valued map, 
$c\colon\R\to\R^k$ is a sufficiently regular vector-valued map, $f_1$ and $f_2$ are continuous 
with $f_1\colon\R \times\R ^k\times M\times[0,\infty)\to\R^k$ and
$f_2\colon\R\times\R^k\times M\times[0,\infty)\to\R^s$ and, in particular, $f_2$ is a tangent 
vector field to $M$, in the sense that for every $(t, p, q,\lambda)\in\R\times\R^k\times M\times[0,\infty)$
it holds that $f_2(t, p, q,\lambda)\in T_qM$. Moreover, all of these maps are assumed to be $T$-periodic, 
$T>0$ given, with respect to the $t$-variable.

By applying fixed point index and degree-theoretic methods, in our main result (Theorem \ref{tuno} below) we deduce a global continuation principle for the $T$-periodic solutions of \eqref{eq:main}. More precisely, assuming a non-$T$-resonance condition on $A$  we provide a condition based on the degree of the vector field $f_2$ averaged along a solution of the first equation of \eqref{eq:main} for $\lambda=0$, that ensures the existence of a connected set of (nontrivial i.e.\ corresponding to $\lambda>0$, see definition \ref{deftrip} below) $T$-periodic solutions of \eqref{eq:main} whose closure is not compact emanating from a specific set $\Gamma$ of $T$-periodic solutions of the unperturbed \eqref{eq:main}.

\smallskip
The arrangement of equations in \eqref{eq:main} may look unusual at first sight but it is pretty natural in 
some contexts. Consider, for instance, a moving point in the plane or space whose radial distance from the origin is governed by a possibly perturbed linear equation (the perturbation itself depending on the position in the plane or space), while its ``angular'' position with respect to any given frame of reference is influenced only by a small forcing. The motion of such a point in the plane could be described by a curve in polar coordinates $t\mapsto\big(r(t),\theta(t)\big)\in\R\X S^1$ (allowing negative $r$ in the usual sense) that is determined by a system of equations as follows:
\begin{equation*}
\left\{ \begin{array}{ll} 
  \dot r = a(t)r + c(t) + \lambda f_1(t, r, \theta,\lambda), & r\in\R \\
  \dot \theta = \lambda f_2(t, r, \theta,\lambda), & \theta\in S^1
    \end{array} \right. \quad
  \lambda\geq 0,
\end{equation*}
where $f_1$ and $f_2$, that represent perturbations, are $T$-periodic in $t$ of the same period as the scalar 
functions $a$ and $c$.

In fact the form of system \eqref{eq:main} is rather flexible. In later sections we will show, throughout some examples, how it lends itself to quite distant applications.

To get a rough idea of the kind of result that we seek, consider the following equation in $\R\X\R$:
\begin{equation}\label{nicexa}
 \left\{\begin{array}{l}
         \dot x = -x +1+\frac{1}{10}\sin t -\lambda\big|y-x\big|,\\
         \dot y = -\lambda\left(\frac{1}{2}+y+2x\sin t\right).
        \end{array}\right.
\end{equation}
The numerical diagram in figure \ref{figuno} shows a portion of the set of points $(\lambda,x_0^\lambda,y_0^\lambda)$ where $(x_0^\lambda,y_0^\lambda)$ is the initial data (for $t=0$) of a $T$-periodic solutions of \eqref{nicexa} corresponding to $\lambda$.
\begin{figure}[ht!]
 \centering
 \includegraphics[width=0.79\linewidth]{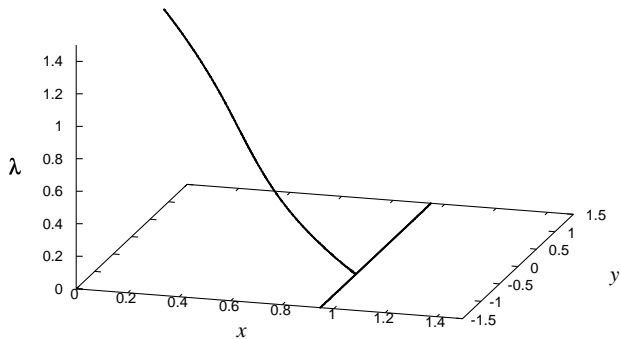}
 \caption{Initial points of $T$-periodic solutions of \eqref{nicexa} }\label{figuno}
\end{figure}

Roughly speaking, Theorem \ref{tuno} yields a connected set $\Gamma$ of ``nontrivial'' $T$-periodic solutions emanating from a subset of initial points as in the figure \ref{figuno}. Later on, we will analyze the behavior of this set in more detail (see Definiton \ref{defst} and Theorem \ref{tdue} below).

\smallskip
On a different tack, observe that system \eqref{eq:main} is a periodic perturbation of a nonautomous 
differential equation. In general, periodic problems of such equations are rather difficult to study,
at least with topological methods. One reason being that it is not easy to compute the fixed point
index of the Poincar\`e $T$-translation operator associated with a nonautonomous equation. The present 
study does this in a particular case. A different situation, but in the same direction was considered 
in \cite{SpSepVar} where perturbations of separated variables equations are tackled by a change of 
variable approach. Here, the difficulties arising from the fact that the unperturbed part of system 
\eqref{eq:main} is nonautonomous are solved thanks to the simple observation that the first equation 
in of \eqref{eq:main} admits a unique $T$-periodic solution when $\lambda=0$ and some assumptions are 
made on $A$.

The technique used here could probably be generalized to the case when the unperturbed first equation of \eqref{eq:main} is nonlinear but admits a finite number of ``nondegenerate'' periodic solutions.  Moreover, our methods open the way to the analysis of coupled differential equations of type \eqref{eq:main} involving delays in the perturbing term (see, e.g., \cite{Bi-Spa-NoDEA,Bi-Spa-TMNA, Fu-Spa-delay} for some recent contributions about similar issues). These lines of research would take us too far from the core topic of the present paper, but will be investigated in future studies.  \smallskip

Finally, we point out that our arguments rely on the notions of degree (or characteristic) of a tangent vector field and of fixed point index of a map on a manifold. For an exposition, we refer to standard texts as, e.g., \cite{difftop,milnor}.

\section{Preliminaries and main result} \label{sec:notation} We begin this section by recalling 
some basic facts and definitions about the function spaces used throughout the paper. 

Let $I\subseteq\R$ be an interval and let $X\subseteq\R^n$.  The set of all $X$-valued 
functions defined on $I$ is denoted by $C(I,X)$. When $I=\R$, we simply write $C(X)$
instead of $C(\R,X)$. Let $T>0$ be given, by $C_T (\R^n)$ we mean the Banach space
of all the continuous $T$-periodic functions $\zeta\colon\R\to\R^n$ whereas $C_T (X)$ denotes 
the metric subspace of $C_T (\R^n)$ consisting of all those $\zeta\in C_T (\R^n)$ that take values 
in $X$. In fact, $X$ can be regarded as a subset of $C_T(X)$ by identifying it with the image of
the closed embedding that associates to any point $q\in X$ the constant function $\bar q(t)\equiv q$.
Also, it is not difficult to prove that $C_T (X)$ is complete if and only if $X$ is closed in $\R^n$.

Let $Y\subseteq \R^k$ and $W\subseteq \R^s$ with $Y\X W\subseteq \R^k\X \R^s\cong \R^n$.  There is 
a natural homeomorphism between $C_T(Y)\X C_T(W)$ and $C_T(Y\X W)$. Then, for the reminder of the
paper we will use $C_T(Y\X W)$ and $C_T(Y)\X C_T(W)$ interchangeably, and denote the elements of 
$C_T(Y\X W)$ as pairs $(x,y)$ with $x\in C_T(Y)$ and $y\in C_T(W)$.

In the sequel, if $\mathfrak{S}$ is a topological space, given $A\subseteq [0,\infty)\X\mathfrak{S}$ and $\lambda\in[0,\infty)$, we will denote the $\lambda$-slice $\{x\in\mathfrak{S}:(\lambda,x)\in A\}$ by the symbol $A_\lambda$. Observe that when $A$ is open in $[0,\infty)\X\mathfrak{S}$ then so is $A_\lambda\subseteq\mathfrak{S}$.

\medskip
Our main result concerns the properties of the set of the $T$-periodic solutions of \eqref{eq:main}. Let 
us recall the following definition:

\begin{definition}\label{deftrip}
  A triple $(\mu, x, y\big) \in [0, +\infty) \X C^1(\R^k\X M)$ is called a $T$-\emph{triple} for \eqref{eq:main} if $(x, y)$ is a $T$-periodic solution of \eqref{eq:main} when $\lambda=\mu$. A $T$-triple $\big(\lambda, x, y\big)$ is called \emph{trivial} if $\lambda=0$.
\end{definition}

Observe that a trivial $T$-triple $(0,x,y)$ has necessarily a specific form: $y$ is clearly constant and $x$ is a $T$-periodic solution of the ODE
\begin{equation} \label{eq:linear-ODE} 
\dot x = A(t) x + c(t),
\end{equation}
where $A$ and $c$ are as in \eqref{eq:main}. 

Let $\Phi=\Phi(t)\in \R^{k\X k}$ be the principal fundamental matrix of equation $\dot x(t)=A(t)x(t)$. Namely, $\Phi\colon\R\to\R^{k\X k}$ is such that
  \begin{equation*}
    \dot \Phi(t) = {A}(t) \Phi(t)\, \textrm{ with }\,  \Phi(0)=I:=I_{k\X k}.
  \end{equation*}
Throughout this section we posit the following non-$T$-resonance assumption on $A$:
\begin{equation}\label{noTres}
 \text{$I-\Phi(T)$ is invertible.}
\end{equation}
When $A(t)\equiv A$ is a constant matrix, we have that $\Phi(t)=e^{tA}$ and the condition \eqref{noTres} can be written as $\det(I-e^{TA})\neq 0$ which holds true if and only if $A$ has no eigenvalues of the form $2n\pi i/T$, $n\in\Z$, $i$ being the imaginary unit.

When the matrix-valued function $A$ has the ``non-$T$-resonance'' property above (see also, e.g., \cite{BS15}) something can be said about the $T$-periodic solutions of \eqref{eq:linear-ODE}, as shown in the following well-known technical lemma whose standard proof is provided only for the sake of completeness:

\begin{lemma} \label{lem:utility-result} 
Assume that $I-\Phi(T)$ is invertible. Then, equation \eqref{eq:linear-ODE} admits a unique 
$T$-periodic solution given by:
\begin{equation} \label{eq:solution} 
\hat x (t):=\Phi(t)\left(
    \big(I - \Phi(T)\big)^{-1}\Phi(T) \int_0^T \Phi(s)^{-1} c(s)ds +\int_0^t \Phi(s)^{-1} c(s)ds\right) \;.
  \end{equation}
\end{lemma}

\begin{proof}
The initial value problem
\begin{equation} \label{IVP} 
  \left\{ \begin{array}{l} \dot x =
        A(t) x + c(t),\\
        x(0)=x_0,
      \end{array} \right.
\end{equation}
admits a unique solution that can be written as
\begin{equation}\label{linsol}
    x(t)=\Phi(t)x_0+\Phi(t) \int_0^t\Phi(s)^{-1} c(s) ds.
\end{equation}
We seek initial values $x_0$ that correspond to $T$-periodic solutions. Hence, we put
\[
 x_0=x(T)=\Phi(T)x_0+\Phi(T) \int_0^T\Phi(s)^{-1} c(s) ds,
\]
and solve for $x_0$. Since $I-\Phi(T)$ is invertible we get
\begin{equation*}
    x_0=\big(I- \Phi(T)\big)^{-1}\Phi(T)\int_0^T \Phi(s)^{-1} c(s )ds.
\end{equation*}
which is the unique initial condition that yields a $T$-periodic solution of \eqref{IVP}.
Susbstituting in \eqref{linsol} we get the assertion.
\end{proof}

Observe that when $A(t)\equiv A$ is a constant matrix, \eqref{eq:solution} assumes a slightly nicer-looking form:
\begin{equation*}
 \hat x (t)=e^{tA}\left((I-e^{TA})^{-1}e^{TA}\int_0^T e^{-sA} c(s)ds+\int_0^t e^{-sA} c(s)ds\right) \;.
\end{equation*}
\smallskip

We introduce two projections $\pi_1\colon [0,\infty)\times C_T(\R^k\times M)\to [0,\infty)\times C_T(\R^k)$ 
and $\pi_2\colon [0,\infty)\times C_T(\R^k\times M)\to [0,\infty)\times C_T(M)$ given by
\begin{equation}\label{proiez}
\left.
\begin{array}{l}
\pi_1(\lambda,x,y):=(\lambda,x)\\
\pi_2(\lambda,x,y):=(\lambda,y)
\end{array}
\right\}\quad\text{for all $(\lambda,x,y)\in [0,\infty)\times C_T(\R^k\times M)$}
\end{equation}
Also, it is useful to introduce the following notation:
Given $\Omega\subseteq [0,\infty)\times C_T(\R^k\times M)$ open, let $\mathcal{O}_1(\Omega)\subseteq\R^k$, 
and $\mathcal{O}_2(\Omega)\subseteq M$ be the open sets given, respectively, by
\[
 \mathcal{O}_1(\Omega):=\big\{p\in\R^k:(0,\cl{p})\in\pi_1(\Omega)\big\},
 \quad\text{and}\quad  
 \mathcal{O}_2(\Omega):=\big\{q\in M:(0,\cl{q})\in\pi_2(\Omega)\big\}.
\]
Actually, $\mathcal{O}_1(\Omega)$ and $\mathcal{O}_2(\Omega)$ can be seen as ``finite dimensional versions'' of the $0$\hbox{-}slices $\big(\pi_1(\Omega)\big)_0$ and $\big(\pi_2(\Omega)\big)_0$, respectively. In fact, one clearly has $\mathcal{O}_1(\Omega)=\big\{p\in\R^k:\cl{p}\in\big(\pi_1(\Omega)\big)_0\big\}$ and $\mathcal{O}_2(\Omega)=\big\{q\in M:\cl{q}\in\big(\pi_2(\Omega)\big)_0\big\}$.

It is also convenient to define the following average vector field:
\begin{equation}\label{mainaverage}
   w(q):=\frac{1}{T}\int_0^T f_2\big(t,\hat x(t),q,0\big)\, dt,\qquad q\in M,
\end{equation}
where $f_2$ and $\hat x$ are as in \eqref{eq:main} and \eqref{eq:solution}, respectively. Clearly, $w$ is tangent
to $M$.

We are now in a position to state our main result; its proof, that requires several preliminary steps, is
postponed to a later section.

\begin{theorem}\label{tuno} Let $A$, $c$, $f_1$ and $f_2$ be as in \eqref{eq:main}, and let $\hat x$ be as in 
\eqref{eq:solution}. Let $w$ be as in \eqref{mainaverage}. Let $\Omega$ be a given 
open subset of $[0,\infty)\times C_T(\R^k\X M)$, and assume that $\deg\big(w,\mathcal{O}_2(\Omega)\big)$ is 
well-defined and nonzero and that $\hat x (0)\in\mathcal{O}_1(\Omega)$. Then, there exists a connected set $\Gamma$ 
of nontrivial $T$-triples for \eqref{eq:main} in $\Omega$ whose closure in $[0,\infty)\times C_T(\R^k\X M)$ meets 
the set
\[
    \big\{(0, \hat x, \cl{q})\in \Omega :   w(q) =0  \big\}
\]
and $\Gamma$ is not contained in any compact subset of $\Omega$.
In particular, if $M$ is a complete manifold and $\Omega = [0,\infty)\X C_T(\R^k\X M)$ then $\Gamma$ is
unbounded.
\end{theorem}

\medskip
We now illustrate the setting of Theorem \ref{tuno} with an explicit example.

\begin{example} Consider the following system of constrained ODEs in $\R^3$:
\begin{equation}\label{ex.2}
  \left\{\begin{array}{l}
      \dot x = -x+\sin(t)-\cos(t)+\lambda f(t,x,y,z),\\ 
      \dot y = \lambda z\big(y+x\sin(t)\big),\\
      \dot z = -\lambda y\big(y+x\sin(t)\big),\\
      y^2+z^2=1.
    \end{array}\right. 
\end{equation}
Where $f\colon\R\X\R^3\to\R$, continuous and $2\pi$-periodic in $t$, is given. System \eqref{ex.2} can be regarded as an ODE on $\R\X S^1$ (here $S^1$ lies in the plane $y,z$) because the time-dependent vector field defined by the second and third equations in \eqref{ex.2} is tangent to $S^1$. Taking $\hat x$ and $w$ as in Theorem~\ref{tuno} 
(with $T=2\pi$), we get $\hat x(t)=-\cos(t)$ and
\[
w(y,z) = \frac{1}{2\pi}\int_0^{2\pi}\begin{pmatrix}
  zy+z\sin(t)\hat x(t)\\
  -y^2-y\sin(t)\hat x(t)
\end{pmatrix}dt =\begin{pmatrix}
  zy\\
  -y^2
\end{pmatrix} .
\]
Clearly $w$ has only two zeros on $S^1$ given by $\mathbf{N}:=(0,1)$ and $\mathbf{S}:=(0,-1)$. 
Consider, for instance, $U_\mathbf{N}:=S_1\setminus\{\mathbf{S}\}$ and let
\[
 \Omega_\mathbf{N}:=[0,\infty)\X C_T(\R\X U_\mathbf{N}).
\]
It is not difficult to see that
\[
\deg\big(w,(\R\X S^1)\cap\Omega_\mathbf{N}\big)=\deg(w,U_\mathbf{N})=1.
\]
Hence, Theorem \ref{tuno} yields a connected set $\Gamma_\mathbf{N}$ of nontrivial $T$-triples for \eqref{ex.2} 
whose closure in $[0,\infty)\times C_T(\R^k\X M)$ is not contained in any compact subset of $\Omega_\mathbf{N}$
and  meets the set
\[
\big\{(0, \hat x, \cl{\mathbf{N}})\in [0,\infty)\X C_T(\R\X S^1) \big\}.
\]
\end{example}


The previous example opens the way to the study of a more significant
situation concerning the coupled differential equations describing 
a simple mechanical system, and the possible connections to a particular class
of Differential-Algebraic Equations (DAEs).

\begin{example}\label{expendolo}
Consider a block of mass $m$ constrained to a unitary circular rail lying on an horizontal plane. Put the origin of the plane in the center of the rail, and suppose the block can slide along the rail without friction and that it is acted upon by a continuous $T$-periodic force of the form $\phi(t,x,y)\left(\begin{smallmatrix}-y\\ x\end{smallmatrix}\right)$ (which is clearly tangent to the constraint). Let us denote by $r(t)$ the modulus of the centripetal force (the vincular reaction of the rail) exerted, at time $t$, on the block. One customary way of writing the equations governing such a system is through the formalism of Differential-Algebraic Equations (DAEs) see, e.g.\ \cite{dae},  as follows:
\begin{subequations}\label{coupled}
\begin{equation}\label{pendolacc}
  \left\{ \begin{array}{l}
            m\ddot x(t) = -r(t) x(t) -y\phi(t,x,y),\\
            m\ddot y(t) = -r(t) y(t) +x\phi(t,x,y),\\
            x(t)^2 + y(t)^2 = 1,
          \end{array}\right. \,\,\, \lambda >0,
\end{equation}
(See also, e.g., \cite[Example 4.7, \S 4]{Gerdin}.)
Consider now a linear oscillator of (fixed) unit mass and elastic constant $k$ subject to a $T$-periodic forcing $c(t)$, and assume, finally, that the acceleration $r(t)/m$ is used as imput by some nonspecified mechanism to influence this oscillator by the addition of a force $f\big(r(t)/m\big)$. Thus this oscillator is described by the following equation:
\begin{equation}\label{oscillacc}
 \ddot z(t) = -k z(t) + c(t) + f\big(r(t)/m\big).
\end{equation}
\end{subequations}
The system under consideration is schematized in figure \ref{fig:expendolo}.

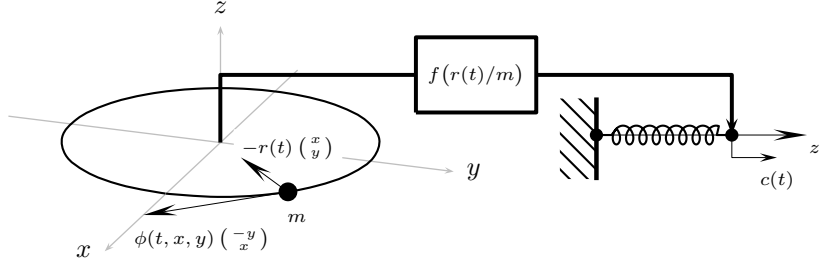
\begin{figure}[h!]
\begin{center}
  \begin{pspicture}(-4.8,-1.45)(6,1.8)
    \psset{unit=1cm}
    \put(-2,0){
    \psset{unit=1.5cm, Alpha=70,Beta=20}
    \pstThreeDCoor[xMin=-2,xMax=3,yMin=-2,yMax=2.2,zMin=0,zMax=1.1,linecolor=lightgray]
    \parametricplotThreeD[fillstyle=none,xPlotpoints=20,linecolor=black,linewidth=0.8pt,plotstyle=curve,algebraic](0,6.28){1.41*cos(t) , 1.41*sin(t), 0}
    \pstThreeDDot[linewidth=0.06](1,1,0)
    \pstThreeDLine[arrowsize=3pt 5,arrowlength=3,linewidth=0.25pt]{*->}(1,1,0)(2,0,0)
    \pstThreeDPut(1.6,1.3,0){\scriptsize $m$}
    \pstThreeDPut(2.4,0.7,0){\scriptsize $\phi(t,x,y)\left(\begin{smallmatrix}-y\\ x\end{smallmatrix}\right)$}
    \rput*(0.6,-0.05){\scriptsize $-r(t)\left(\begin{smallmatrix}x\\ y\end{smallmatrix}\right)$}
    \pstThreeDLine[arrowsize=3pt 5,arrowlength=3,linewidth=0.25pt]{*->}(1,1,0)(0.3,0.3,0)
    }
    \psline[linewidth=1.5pt](-2,0)(-2,0.9)(0.6,0.9)
    \psline[linewidth=1.5pt](0.6,0.4)(0.6,1.4)(2.2,1.4)(2.2,0.4)(0.6,0.4)
    \rput(1.4,0.9){\scriptsize $f\big(r(t)/m\big)$}
    \psline[linewidth=1.5pt]{->}(2.2,0.9)(4.8,0.9)(4.8,0.1)
    \psline[arrowsize=3pt 5,arrowlength=3,linewidth=0.25pt]{->}(3,0.1)(5.8,0.1)
    \pspolygon[linecolor=white,fillstyle=vlines](3,-0.5)(3,0.6)(2.5,0.6)(2.5,-0.5)
    \psline[linewidth=1.7pt](3,-0.5)(3,0.6)
    \uput[dr](5.7,0.1){\scriptsize $z$}
    \psdot[linewidth=1.5pt](3,0.1)
    \psdot[linewidth=1.5pt](4.8,0.1)
    \pscoil[coilarm=.2cm,linewidth=0.8pt,coilwidth=.25cm]{*-*}(3,0.1)(4.8,0.1)
    \psline[arrowsize=2pt 5,arrowlength=2,linewidth=0.25pt]{->}(4.8,0.1)(4.8,-0.2)(5.4,-0.2)
    \uput[d](5.4,-0.2){\scriptsize $c(t)$}
  \end{pspicture}
\end{center}
\caption{The system of Example \ref{expendolo}}\label{fig:expendolo}
\end{figure}

We are interested in the set of $T$-periodic solutions of the coupled system \eqref{coupled} as $m>0$ varies. 

Setting $\lambda^2=1/m$, $\zeta=\dot z$, $\dot x=\lambda u$, $\dot y=\lambda v$, and assuming that $f$ is homogeneous of degree $\alpha$, the equations governing \eqref{coupled} become:
\begin{equation} \label{eq:DAE-rew}
  \left\{ \begin{array}{l}
            \dot z(t) = \zeta (t),\\
            \dot \zeta(t) = - z(t) + c(t) + \lambda^{2\alpha} f\big(r(t)\big),\\
            \dot x(t) = \lambda u(t),\\
            \dot y(t) = \lambda v(t), \\
            \dot u(t) = \lambda \big( -r(t) x(t) -y(t)\phi(t,x,y)\big),\\
            \dot v(t) = \lambda \big( -r(t) y(t) +x(t)\phi(t,x,y)\big),\\
            x^2(t) + y^2(t) = 1.
          \end{array}\right. \,\,\, \lambda >0,
    \end{equation}
System \eqref{eq:DAE-rew} is actually a DAE with differentiation index 2 (see \cite{dae}). Differentiating with respect to $t$ the ``algebraic'' constraint $x(t)^2 + y(t)^2=1$, we get $2\dot x(t)x(t)+2\dot y(t)y(t)=0$. Using \eqref{eq:DAE-rew} and dividing by $2\lambda>0$, we get $x(t)u(t)+y(t)v(t)=0$. Differentiationg again and using the other relations in \eqref{eq:DAE-rew} we get that any solution $t\mapsto\big(x(t),y(t),u(t),v(t),r(t),z(t),\zeta(t)\big)$ of \eqref{eq:DAE-rew} satisfies
\[
 u(t)^2+v(t)^2-r(t)\big(x(t)^2 + y(t)^2\big)=0.
\]
Whence,
\[
 \dot r(t)= 2\big(\dot u(t)u(t)+\dot v(t)v(t)\big).
\]
Thus, we have a differential equation for $r$: 
\[
 \dot r(t)= -2\lambda\Big(\big( r x +y\phi(t,x,y)\big)u+   \big(r y -x\phi(t,x,y)\big)v\Big).
\]
Pluggiung this differential equation in \eqref{eq:DAE-rew} we get an equivalent equation on $\R^2\X M\subseteq\R^7$, given by:
\begin{equation} \label{DAEManifold}
  \left\{ \begin{array}{l}
            \dot z(t) = \zeta (t),\\
            \dot \zeta(t) = - z(t) + c(t) + \lambda^{2\alpha} f\big(r(t)\big),\\
            \dot x(t) = \lambda u(t),\\
            \dot y(t) = \lambda v(t), \\
            \dot u(t) = \lambda \big( -r(t) x(t) -y(t)\phi(t,x,y)\big),\\
            \dot v(t) = \lambda \big( -r(t) y(t) +x(t)\phi(t,x,y)\big),\\
            \dot r(t)= -2\lambda\Big(\big( r x +y\phi(t,x,y)\big)u+   \big(r y -x\phi(t,x,y)\big)v\Big),
          \end{array}\right. \,\,\, \lambda >0,
    \end{equation}
where $M$ is the $2$-dimensional submanifold of $\R^5$ defined by the constraints obtained above, namely:
 \[
  M:=\Big\{(X,Y,U,V,R)\in\R^5: X^2+Y^2=1,\, XU+YV=0,\, U^2+V^2-R=0 \big\},
 \]
here we use capital letters to denote the coordinates in $\R^5$ in order to distinguish from the dependent variables in \eqref{eq:DAE-rew} and \eqref{DAEManifold}.
    
Equation \eqref{DAEManifold} can be extended to the case $\lambda\in [0,\infty)$, although its physical meaning is lost when $\lambda=0$. However, Theorem \ref{tuno} can be used to investigate the limiting behaviour of the system \eqref{coupled} and to gain insights into the topological structure of its $T$-periodic solutions.

Finally, note that \eqref{eq:DAE-rew}, after index reduction, falls within the framework considered in \cite{SpaDAE, BisDAE} (see also \cite{Bi-Spa-TMNA, Bi-Spa-EJQTDE} for the case of similar DAEs involving delay). 
\end{example}

\section{Fixed point index and Poincar\'e translation operator} \label{sec:poincare}
 Throughout this section we will assume that $f_1$ and $f_2$ of equation 
\eqref{eq:main} are locally Lipschitz so that the initial value problem 
\begin{equation} \label{eq:cauchy} 
  \left\{ \begin{array}{l}
     \dot x = {A}(t)x + c(t) + \lambda f_1(t, x, y,\lambda), \\
     \dot y = \lambda f_2(t, x, y,\lambda),\\
     (x(0), y(0))= (p, q),
          \end{array} \right.
\end{equation}
admits a unique solution for $\lambda \geq 0$. 

We denote by $P^\lambda_\tau (p,q)$ the value of the solution of \eqref{eq:cauchy} at time $\tau$, when 
defined:  That is, $P_\tau^\lambda$ is the so-called Poincar\'e $\tau$-translation operator associated 
with \eqref{eq:cauchy}.  Well-known properties of differential equations show that the domain
\begin{equation}
\begin{split}\label{defD}
  \mathcal{D}=\Big\{ (\ell, p,q) \in [0,\infty)\X\R^k\X M : &\text{ the solution of \eqref{eq:cauchy} 
  for $\lambda=\ell$}\\
   &\qquad\quad\text{is continuable up to $t=T$}  \Big\}
\end{split}
\end{equation}
of the map $(\lambda,p,q)\mapsto P^\lambda_T(p,q)$, is open in $[0,\infty)\X \R^k\X M$. Consequently,
fixed any $\lambda\in [0,\infty)$, the $\lambda$-slice of $\mathcal{D}$
\begin{equation*}
 \mathcal{D}_\lambda = \Big\{(p,q)\in\R^k\X M : (\lambda,p,q)\in \mathcal{D}\Big\},
\end{equation*} 
is the domain of the map $(p,q)\mapsto P^\lambda_T(p,q)$, is open in $\R^k\X M$.

As the functions $A$, $c$, $f_1$ and $f_2$ are $T$-periodic in $t$, it is well-known 
that $T$-periodic solutions of \eqref{eq:main} (for any given $\lambda\geq 0$) correspond 
bijectively to fixed points of the $T$-translation operator $P_T^\lambda$, in the sense 
that an initial condition $(p,q)$ of \eqref{eq:cauchy} yields a $T$-periodic solution 
if and only if $(p,q)$ is a fixed point of $P_T^\lambda$.  
\smallskip

Since our purpose is to study the $T$-periodic solutions of \eqref{eq:main}, it makes 
sense to investigate the fixed point index of the $T$-translation operator $P_T^\lambda$.
The following theorem gives a useful formula; it is a major step towards Theorem \ref{tuno}
and the main result of this section.

\begin{theorem}\label{th:formulaind}
  Let $U\X V\subseteq \R^k\X M$ be an open subset such that $\hat
  x(0)\notin \D U$ and $w^{-1}(0)\cap \D V=\emptyset$ where $w$ is
  defined in \eqref{mainaverage}. Let $P_T^\lambda$ be the
  $T$-translation operator associated with \eqref{eq:main}. Then, for
  sufficiently small values of $\lambda>0$
  \begin{equation} \label{eq:disaccoppind} \left|\ind
      \big( P_T^\lambda, U\X V\big) \right| =\caratt{U}(\hat x (0))
    \big|\deg \big( w, V\big)\big|.
  \end{equation}
where $\caratt{U}$ denotes the characteristic function of $U$.
\end{theorem}

The proof of this theorem requires some preliminary steps. Let us begin with a simple formula for 
the fixed point index of the $T$-translation operator associated to equation \eqref{eq:linear-ODE}.
More precisely, let $\mathcal{F}_T(p)\in\R^k$ be the value, at time $t=T$, of the maximal 
solution of \eqref{eq:linear-ODE} that satisfies $x(0)=p$. Since $A$ and $c$ are bounded functions 
(they are assumed $T$-periodic), it is well known that all solutions of \eqref{eq:linear-ODE} are
continuable in $\R$; that is, the domain of the function $\mathcal{F}_T$ is the whole $\R^k$.

The formula we need is the content of the following preliminary lemma:

\begin{lemma} \label{lem:first} 
Let $\mathcal{F}_T$ be as above, and let $U\subseteq \R^k$ be an open and such that $\hat x(0)\notin\D U$, 
then
\begin{equation}
    \left|\ind \big(\mathcal{F}_T, U\big)\right| = \caratt{U}(\hat x (0)).
\end{equation}
where, as in Theorem \ref{th:formulaind}, $\caratt{U}$ denotes the characteristic function of $U$.
\end{lemma}

\begin{proof}
  Let $\Phi$ be as in Lemma \ref{lem:utility-result}.  Then, by \eqref{linsol},
  \begin{equation*}
    \mathcal{F}_T(p) = \Phi(T)p + \Phi(T)\int_0^T\Phi(s)^{-1}c(s)ds,
  \end{equation*}
  for all $p\in\R^k$. Taking the Fr\'echet derivative of $\mathcal{F}_T$ at
  $p_0:=\hat x(0)$, we get
  \begin{equation*}
    \mathcal{F}'_T(p_0)h = \Phi(T), \,\, \forall h\in T_{p_0}\R^k\cong \R^k.
  \end{equation*}
  Hence, since $\Phi(T)$ is nonsingular, by the properties of the fixed point index we have that
  \begin{equation*}
    \ind (\mathcal{F}_T, U)= \left\{\begin{array}{ll} \sign \big(\det \Phi(T)\big),
        & \text{ if } p_0\in U,\\
        0, & \text{ if } p_0\notin U.\\ 
      \end{array} \right.
  \end{equation*}
The assertion follows.
\end{proof}

We will need a slight extension of a result of \cite{fupespa} that concerns the translation 
operator associated to equations of the following form:
\begin{equation} \label{eq:eq-comodo} 
\dot y = \lambda \varphi(t,y,\lambda),\quad \lambda \geq 0,
\end{equation}
where $\varphi\colon\R\X N\X [0,\infty)\to\R^n$ is a $C^1$ map, tangent to a manifold $N\subseteq\R^n$ 
in the sense that $\varphi(t,q,\lambda)\in T_qN$ for all $(t,q,\lambda)\in\R\X N\X [0,\infty)$, which is 
$T$-periodic in the first variable. 

We will denote by $Q_{\varphi,T}^\lambda$ the $T$-translation operator associated with \eqref{eq:eq-comodo}. 
Namely, $Q_{\varphi,T}^\lambda(q)$, when defined, will be the value for $t=T$ of the (unique) solution of 
\eqref{eq:eq-comodo} satisfying $y(0)=q$. As discussed above for the case of $P_T^\lambda$ one can prove 
that the domainn of the map $(\lambda,q)\mapsto Q_{\varphi,T}^\lambda(q)$ is open in $[0,\infty)\X N$. We 
also define the average vector field
\begin{equation}\label{avgVF}
  \mathrm{w}_\varphi (q) =\frac{1}{T} \int_0^T \varphi(t, q, 0)\, dt,\qquad q\in N,
\end{equation}
which is clearly tangent to $N$.

\begin{theorem}\label{T.3.11}
Let $f$, $\mathrm{w}_\varphi$ and $N$ be as above. Let $U$ be a relatively compact open subset of $N$ and assume 
that $(\mathrm{w}_\varphi,U)$ is admissible for the degree.  Then, there exists $\lambda_0>0$ such that, for 
$0<\lambda\leq\lambda_0$, the $T$-translation operator $Q_{\varphi,T}^\lambda$ associated with \eqref{eq:eq-comodo} 
is defined on $U$, fixed point free on $\D U$ and
\begin{equation*}
    \ind \big( Q^\lambda_{\varphi,T},U\big) = \deg(-\mathrm{w}_\varphi ,U).
\end{equation*}
\end{theorem}

\begin{proof}
 The same argument of the proof of Theorem 3.11 in \cite{fupespa} applies to this more general assertion with 
 only minimal adaptation. 
\end{proof}

The following fact is a direct consequence of the properties of the fixed point index of tangent
vector fields on manifolds. Its simple proof is left to the reader.

\begin{proposition} \label{prop:1} 
Let $X$ and $Y$ be differentiable manifolds and let
Let $F_1\colon A\subseteq X\to Y$ and $F_2\colon B\subseteq Y\to Y$ be continuous maps. Let $U\subseteq A$ 
and $V\subseteq B$ be open subsets and assume, $F_1$ and $F_2$ admissible for the fixed point index in $U$ 
and $V$, respectively. Then $\big(F_1 \X F_2\big)(x, y) := (F_1(x), F_2(y))$ is admissible for the fixed 
point index in $U\X V$ and
\begin{equation} \label{eq:indice-disaccoppiato} \ind \big( F_1\X
  F_2, U\X V\big) =\ind \big( F_1, U\big) \ind\big( F_2, V\big).
\end{equation}
\end{proposition}

\smallskip
Taken together, Theorem \ref{T.3.11} and Proposition \ref{prop:1} yield the following result
about the $T$-translation operator of decoupled systems:

\begin{corollary} \label{cor:disaccoppiato} 
Let $A$, $c$ and $f_2$ be as in \eqref{eq:main}. Consider the following system of equations
\begin{equation} \label{eq:disaccoppiato} 
  \left\{
    \begin{array}{l}
        \dot x = A(t)x + c(t),\\
        \dot y = \lambda f_2\big(t, \hat x(t), y, \lambda\big),
    \end{array}
  \right. \,\, \lambda \geq 0
\end{equation}
where $\hat x$ is the unique $T$-periodic solution of \eqref{eq:linear-ODE}. Let $\mathcal{Q}_T^\lambda$ be the 
$T$-translation operator associated with \eqref{eq:disaccoppiato}, and let $w$ be given by \eqref{mainaverage}.
Given an open $U\X V\subseteq \R^k\X M$ such that $\hat x(0)\notin\D U$ and $w^{-1}(0)\cap\D V=\emptyset$,
 we have, for sufficiently small values of $\lambda\geq 0$,
  \begin{equation} \label{eq:disaccoppiato-ind} 
  \left|\ind \big(\mathcal{Q}_T^\lambda, U\X V\big) \right| 
       =\caratt{U}\big(\hat x (0)\big)\,\big|\deg(w, V)\big|.
  \end{equation}
\end{corollary}

\begin{proof}
Let $\mathcal{F}_T$ be as in Lemma \ref{lem:first}. Let $\hat x$ be as in \eqref{eq:solution} and define, for 
any $t\in\R$ and $q\in M$ the map $\varphi\colon\R\X M\to\R^s$ given by 
\[
\varphi(t,q,\lambda):=f_2\big(t, \hat x(t), q, \lambda\big).
\]
Hence, $\mathrm{w}_\varphi = w$, with $\mathrm{w}_\varphi$ and $w$ given by \eqref{avgVF} and \eqref{mainaverage}, respectively.

Let $Q_{\varphi,T}^\lambda$ be as in Theorem \ref{T.3.11}. Since equations \eqref{eq:disaccoppiato} are completely decoupled (they are independent of each other) it follows that 
  \begin{equation*}
    \mathcal{Q}_T^\lambda = \mathcal{F}_T\X Q_{\varphi,T}^\lambda.
  \end{equation*}
The assertion follows by Proposition~\ref{prop:1}, Theorem~\ref{T.3.11} and Lemma~\ref{lem:first}.
\end{proof}

Observe that a fixed point of $\mathcal{Q}_T^\lambda$ is necessarily of the form $\big(\hat x(0),q\big)$, 
with $q$ a fixed point of the $T$-translation operator associated to equation
\[
 \dot y = \lambda f_2\big(t, \hat x(t), y, \lambda\big).
\]
Thus, the set of fixed points of $\mathcal{Q}_T^\lambda$ coincides with that of the $T$-translation operator 
associated with the following weakly coupled system:
  \begin{equation*}
    \left\{\begin{array}{l}
        \dot x = A(t)x + c(t),\\
        \dot y = \lambda f_2(t, x, y,\lambda),
      \end{array}\right. \,\, \lambda \geq 0.
  \end{equation*}

\smallskip
This remark alone, however, is not enough to prove Theorem \ref{th:formulaind} because \eqref{eq:main} are 
coupled when $\lambda>0$. To overcome this difficulty, we will use a homotopy argument as shown below.

\begin{proof}[Proof of Theorem \ref{th:formulaind}.]
  Let us consider the following system of coupled equations depending
  on two parameters:
  \begin{equation} \label{eq:two-param} 
    \left\{\hspace{-1mm}
     \begin{array}{l}
        \dot x = A(t)x + c(t) + \lambda \mu f_1(t, x, y, \lambda),\\
        \dot y = \lambda\Big(\mu f_2\big(t,\hat x(t),y,\lambda\big)+(1-\mu)f_2\big(t,x,y,\lambda\big)\Big),
      \end{array}\right.
     \text{$\lambda\geq 0$ and $\mu\in [0, 1]$},
  \end{equation} 
where $\hat x$ is the unique $T$-periodic solution of \eqref{eq:linear-ODE}. 
Consider the map $H$, with domain $\mathcal{D}^H\subseteq\R\X \big(\R^k\X M\big) \X\R$, taking values in 
$\R^k\X M$, and defined by
\begin{equation*}
  H(\lambda, p, q, \mu)
  =\Big(x_{\lambda,\mu}\big(p, q, T\big),
  y_{\lambda,\mu}\big(p, q,  T\big)\Big),
\end{equation*}
where $t\mapsto\big(x_{\lambda,\mu}(p, q,t),y_{\lambda,\mu}(p,q,t)\big)$ denotes the unique maximal
solution of the initial-value problem for the system \eqref{eq:two-param} supplemented by the initial 
conditions $x(0)=p$ and $y(0)=q$, $(p, q)\in \R^k\X M$. Well-known properties of differential equations
imply that $\mathcal{D}^H$ is an open set.

Clearly, if we consider $\lambda=\tilde\lambda$ and $\mu=\tilde\mu$ to be fixed parameters, the map 
$(p, q)\mapsto H(\tilde\lambda,p,q,\tilde\mu)$ is the $T$-translation operator associated with
\eqref{eq:two-param} for $\lambda=\tilde\lambda$ and $\mu=\tilde\mu$. As such, its domain is an
open subset of $\R^k\X M$. It is convenient to put $H^\lambda(\mu, p, q):=H(\lambda,p,q,\mu)$ whenever
defined.

Next, we prove that for fixed $\lambda >0$ small enough, $(\mu,p,q)\mapsto H^\lambda(\mu,p,q)$ defines 
an admissible homotopy on $\cl{U\X V}$ with parameter $\mu\in [0, 1]$.  More precisely, we show that:

\smallskip\noindent\textbf{Claim~1.} 
There exists $\lambda_*>0$ such that for each fixed $\lambda\in (0,\lambda_*]$ the homotopy 
$H^{\lambda}\colon [0,1]\X \cl{U\X V}\to\R^k\X M$ is admissible.

To prove the claim it is enough to show there exists $\lambda_\ast >0$, with the property that for 
each $\lambda\in(0,\lambda_\ast]$, the set
\begin{equation*}
 \mathbf{F}_{\lambda}=\big\{(p, q)\in \cl{U\X V} \; :\; H^\lambda(\mu,p,q)=(p,q),\,
           \text{for some $\mu\in[0,1]$}\big\},
\end{equation*}
which is compact being a closed subset of $H\big([0,\varepsilon]\X\cl{U\X V}\X [0,1]\big)$, is 
contained in $U\X V$.

Suppose by contradiction that this is not the case, i.e., that such a choice of $\lambda_*$ cannot be 
done. Then there are sequences $\{\lambda_n\}\in (0,\infty)$, with $\lambda_n \to 0$ as $n\to+\infty$, 
$\{\mu_n\}\subseteq [0,1]$ and $\{(p_n,q_n)\}\subseteq \D U \X \D V$ such that
\begin{equation}\label{eq:fixpt-H}
  H^{\lambda_n}(\mu_n , p_n, q_n)=(p_n, q_n)\in \textbf{F}_{\lambda_n}.
\end{equation}
For compactness reasons we can assume, as $n\to +\infty$, that up to a subsequence $\mu_n\to\mu_0\in[0,1]$
and $(p_n,q_n)\to(p_0,q_0)\in\D U\X\D V$. Let us denote by $(x_n,y_n)$ the $T$- periodic solution of 
\eqref{eq:two-param} corresponding to $(\mu_n,\lambda_n)$ and starting from $(p_n,q_n)$ at time $t=0$. Due 
to the continuous dependence on initial data, it follows that $x_n\to\hat x$ and $y_n(t)\to y_0(t)\equiv q_0$ 
uniformly on $[0, T]$. In particular we have that $\hat x(0)= p_0$.

Moreover, from the second equation in \eqref{eq:two-param} it follows that
\begin{multline*}
 0=y_n(T)-y_n(0)\\
  =\lambda_n \int_0^T \big[\mu_n f_2\big(t, \hat x(t)  ,y_n(t),\lambda_n \big) 
                             + (1- \mu_n) f_2\big(t, x_n(t),y_n(t),\lambda_n \big) \big]\, dt.
\end{multline*}
So, dividing by $\lambda_n>0$, and taking $n\to +\infty$,  we get
\begin{equation*}
 0= \int_0^T\big[\mu_0 f_2(t,\hat x(t), q_0,0) + (1-\mu_0) f_2(t,\hat x(t), q_0,0)\big]\, dt,
\end{equation*}
whence
\begin{equation*}
  0= \int_0^T f_2 \big(t, \hat  x(t), q_0,0\big) dt = w(q_0).
\end{equation*}
This is a contradiction since we assumed that $w(q)\neq0$ for every $q\in \D V$.

\smallskip\noindent\textbf{Claim~2.} Identity \eqref{eq:disaccoppind} holds.

As a consequence of Claim 1 we have that there exists $\lambda_*>0$ so small that $H^\lambda$ is admissible 
for all $\lambda\in(0,\lambda_*]$.  Then, by the homotopy invariance property, we have that
\begin{equation}\label{eq:ind-deg-relation}
 \begin{aligned}
  \ind(P^\lambda_T, U \X V) &=\ind(H^\lambda(0,\cdot,\cdot),U\X V) \\
                            &=\ind(H^\lambda(1,\cdot,\cdot),U\X V) 
                             =\ind( \mathcal{Q}^\lambda_T,U\X V)
    \end{aligned}
  \end{equation}
where $\mathcal{Q}_T^\lambda$ is the $T$-translation operator of Corollary \ref{cor:disaccoppiato}. Using the 
chain of identities \eqref{eq:ind-deg-relation} along with Corollary~\ref{cor:disaccoppiato}, we get
  \begin{equation*}
    \ind( P^\lambda_T,U \X V) = \ind(\mathcal{Q}^\lambda_T, U \X V) 
                              = \caratt{U}\big(\hat x(0)\big)\, \big|\deg(w, U)\big|,
  \end{equation*}
and the assertion follows.
\end{proof}

\section{Proof of the main result}

In order to prove Theorem \ref{tuno}, it is convenient to consider first what we can call its `finite dimensional 
version'. A crucial notion is the following:
    
\begin{definition}\label{defst}
  Let ${A}$, $f_1$, $f_2$, and $c$ be as in \eqref{eq:main}.  An element $(\lambda,p,q)\in [0,\infty)\X\R^k\X M$ 
  is said to be a \emph{starting triple} for \eqref{eq:main} if the initial value problem \eqref{eq:cauchy} 
  admits a $T$-periodic solution. A starting triple $(\lambda,p,q)$ is called \emph{trivial} if $\lambda=0$.
\end{definition}

In other words, a starting triple $(\lambda,p,q)$ is such that $(p,q)$ is an initial point for a $T$-periodic
solution of \eqref{eq:main} corresponding to $\lambda$.

In what follows, the set of all starting triples for \eqref{eq:main} is denoted by $\mathcal{S}$, and the 
subset of nontrivial ones is called $\mathcal{N}$.

An immediate consequence of the continuous dependence on data is that the closure $\cl{\mathcal{N}}$ of 
$\mathcal{N}$ in $[0,\infty)\X\R^k\X M$ is actually a subset of $\mathcal{S}$. Observe now that $\mathcal{S}$ 
is closed in the set $\mathcal{D}\subseteq [0,\infty)\X\R^k\times M$ introduced in \eqref{defD}, even if it might not be so in $[0,\infty)\X\R^k\X M$. Thus, $\cl{\mathcal{N}}$ is closed also in $\mathcal{D}$. Clearly, $\mathcal{D}$ is open in $\R^k\times M$, hence it is locally compact. Thus $\cl{\mathcal{N}}$, as a closed subset of $\mathcal{D}$, is locally compact as well.

The technical result below characterizes the elements of the $0$-slice $(\cl{\mathcal{N}})_0$.

\begin{lemma}\label{Sclos}
Let $A$, $c$, $f_1$ and $f_2$ be as in \eqref{eq:main}, and let $\hat x$ be as in \eqref{eq:solution}. Let 
also $w$ be as in Theorem \ref{tuno}. Assume, that $f_1$ and $f_2$ are locally Lipschitz in $x$ and $y$. If 
$(0,p,q)\in[0,\infty)\X\R^k\X M$ belongs to the closure $\cl{\mathcal{N}}$ of $\mathcal{N}$ in 
$[0,\infty)\X\R^k\X M$, then $p=\hat x(0)$ and $w(q)=0$.
\end{lemma}
\begin{proof}
 It is enough to show that if $\{(\lambda_i,p_i,q_i)\}_{i\in\N}\subseteq\mathcal{N}$ is a sequence with
 $\lambda_i> 0$ and $(\lambda_i,p_i,q_i)\to (0,p_0,q_0)$ then $p_0=\hat x(0)$ and $w(q_0)=0$.
 
 Denote by $(x_i,y_i)$ the (unique) solution of
 \begin{equation*}
\left\{ \begin{array}{l} 
  \dot x = A(t)x + c(t) + \lambda_i f_1(t, x, y,\lambda_i), \\
  \dot y = \lambda_i f_2(t, x, y,\lambda_i),\\
  x(0)=p_i,\\
  y(0)=q_i.
         \end{array} \right. \quad
  \lambda\geq 0,
\end{equation*}
By continuous dependence, $\{(x_i,y_i)\}$ converges uniformly to $\big(\hat x,\cl{q}_0\big)$. Thus, 
$p_0=\hat x(0)$. Let us prove that $w(q_0)=0$. Clearly,
 \[
  0=y_i(T)-y_i(0)=\lambda_i\int_0^Tf_2\big(t,x_i(t),y_i(t),\lambda_i\big)\,dt, \quad\text{for all $i\in\N$}.
 \]
Since $\lambda_i\neq 0$, we get $0=\int_0^Tf_2\big(t,x_i(t),y_i(t),\lambda_i\big)\,dt$, passing to the limit for 
$\lambda_i\to 0$, we get $0=\int_0^Tf_2\big(t,\hat x(t),q_0,0\big)\,dt=w(q_0)$, as desired.
\end{proof}

Similarly to what was done for Theorem \ref{tuno}, we introduce two projections $\mathrm{pr}_1 \colon [0,\infty)\X\R^k\times M\to [0,\infty)\times \R^k$ and $\mathrm{pr}_2 \colon [0,\infty)\X\R^k\times M\to [0,\infty)\times M$ given by
\[
\left.
\begin{array}{l}
\mathrm{pr}_1 (\lambda,p,q):=(\lambda,p)\\
\mathrm{pr}_2 (\lambda,p,q):=(\lambda,q)
\end{array}
\right\}\quad\text{for all $(\lambda,p,q)\in [0,\infty)\X\R^k\X M$}.
\]
So that, given an open subset $\mathcal{U}$ of $[0,\infty)\X\R^k\X M$, the $0$-slices $\big(\mathrm{pr}_1(\mathcal{U})\big)_0$ and $\big(\mathrm{pr}_2(\mathcal{U})\big)_0$, are open in $\R^k$ and $M$, respectively. 

To help clarify the relations between the different projections introduced so far, consider the following commutative diagram:
{\small
\begin{equation}\label{diagramma}
\begin{CD}
 [0,\infty)\X C_T(\R^k)   @<{\pi_1}<< [0,\infty)\X C_T(\R^k\X M) @>{\pi_2}>> [0,\infty)\X C_T( M)\\
 @AAA                                 @AAA                             @AAA         \\
 [0,\infty)\X\R^k @<{\mathrm{pr}_1}<<  [0,\infty)\X \R^k\X M @>{\mathrm{pr}_2}>>   [0,\infty)\X M
\end{CD}
\end{equation}
}
where the vertical arrows, from left to right, are defined as the closed embeddings $(\lambda,p)\mapsto(\lambda,\cl{p})$, $(\lambda,p,q)\mapsto(\lambda,\cl{p},\cl{q})$ and $(\lambda,q)\mapsto(\lambda,\cl{q})$.

\medskip 
The main result concerning the starting triples for \eqref{eq:main} is Theorem \ref{tdue} below. The argument
of the proof follows, with some adaptations, that of \cite[Thm.\ 3]{SpSepVar}, see also \cite[Thm.\ 3.1]{fuspa96}. 
  
\begin{theorem}\label{tdue} 
Let $A$, $c$, $f_1$ and $f_2$ be as in \eqref{eq:main}, and let $\hat x$ be as in \eqref{eq:solution}. 
Assume, as in Theorem~\ref{th:formulaind} that $f_1$ and $f_2$ are locally Lipschitz in $x$ and $y$.
Define $w$ as in Theorem~\ref{tuno}. Let $\mathcal{U}$ be a given open subset of $\mathcal{D}$, and assume 
that $\deg\big(w,\big(\mathrm{pr}_2 (\mathcal{U})\big)_0\big)$ is well-defined and nonzero and that 
$\hat x(0)\in\big(\mathrm{pr}_1 (\mathcal{U})\big)_0$.  
Then, there exists a connected set $G$ of nontrivial starting triples for \eqref{eq:main} in $\mathcal{U}$ 
whose closure in $\mathcal{D}$ meets the set
  \begin{align*}
    Z:=\big\{(0, \hat x(0), q)\in \mathcal{U} :   w(q) =0  \big\}
  \end{align*}
  and is not contained in any compact subset of $\mathcal{U}$.
\end{theorem}

The diagram shown in figure~\ref{figuno} illustrates the situation described in Theorem~\ref{tdue} in the case of equation \eqref{nicexa}.  More specifically, one directly sees that $\hat x(t)= \frac{1}{20}(\sin t-\cos t)+1$, so that $\hat x(0)=\frac{19}{20}$ and (following \eqref{mainaverage}) that
\[
 w(q)=\frac{1}{2\pi}\int_0^{2\pi}-\left(\frac{1}{2}+q+2\hat x(t)\sin t\right) dt= \pi\left(2q+\frac{11}{10}\right).
\]
Hence, by Theorem \ref{tdue}, there is a connected set of nontrivial starting points emanating from $\left\{(0,\frac{19}{20},-\frac{11}{20})\right\}$. In figure~\ref{figdue} we exhibit the projections of the portion of $\Gamma$ considered in figure~\ref{figuno} on the planes $xy$ and $y\lambda$.
\begin{figure}[ht!]
 \begin{tabular}{c@{\hspace{5pt}}c}
  \subfigure[\scriptsize Projection on the $xy$ plane]{\includegraphics[width=0.48\linewidth]{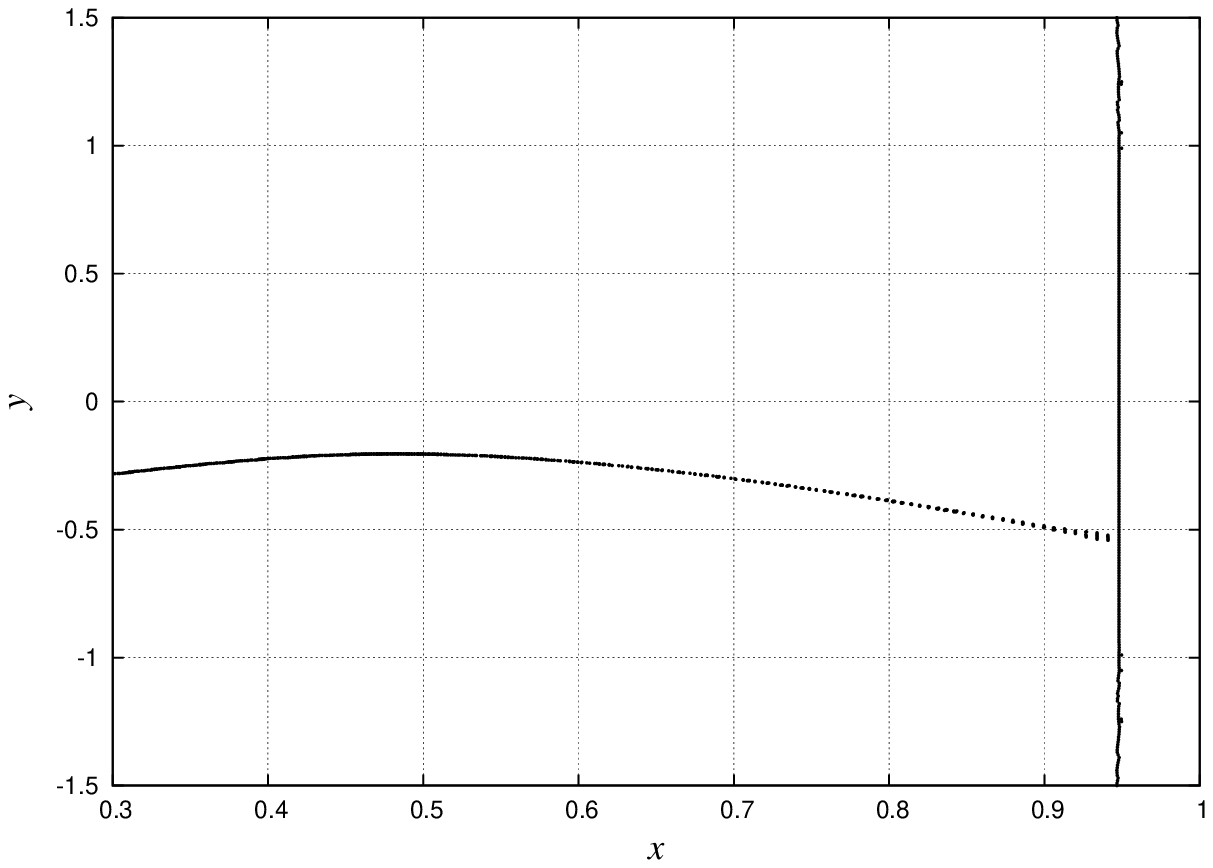}} &
  \subfigure[\scriptsize Projection on the $y\lambda$ plane]{\includegraphics[width=0.48\linewidth]{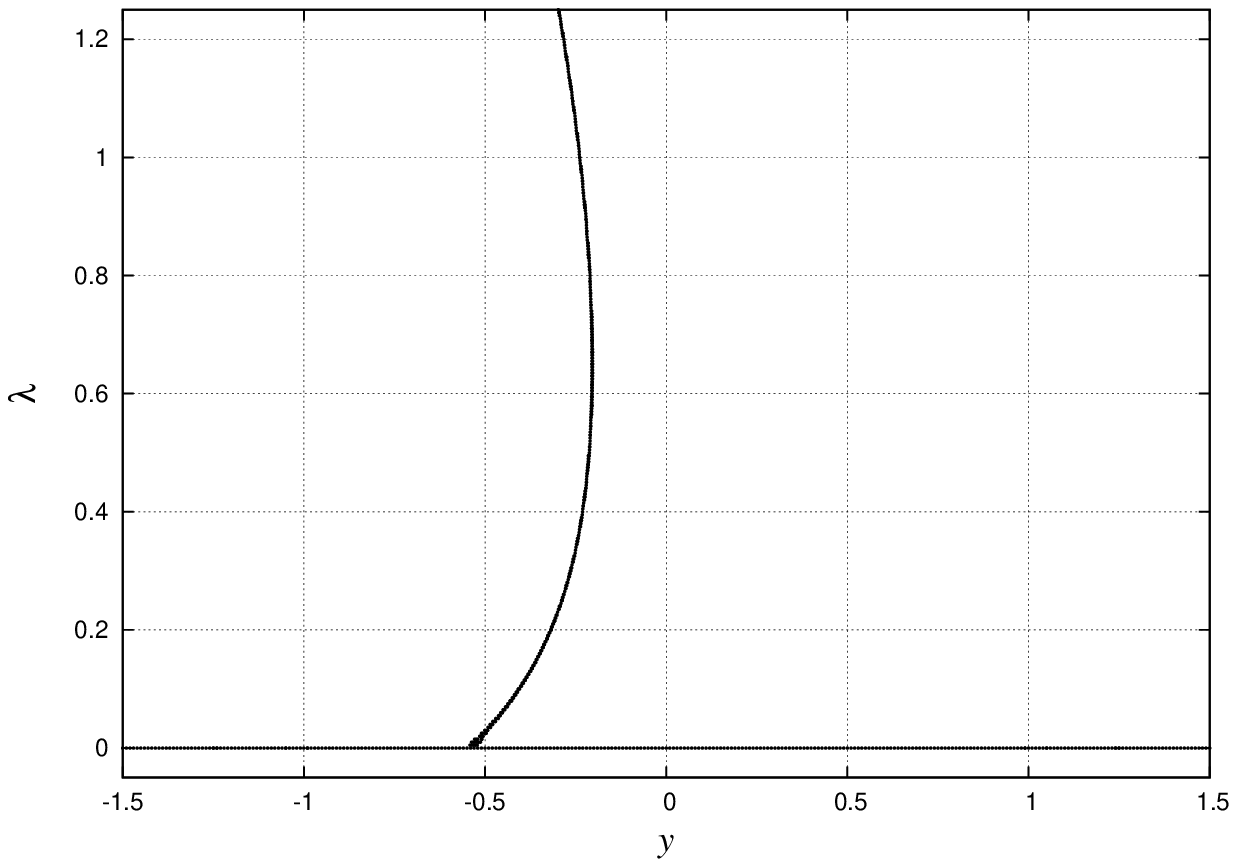}}
 \end{tabular}
\caption{Projections of $T$-periodic solutions of \eqref{nicexa}}\label{figdue}
\end{figure}

Before providing the proof of Theorem \ref{tdue} we recall the following well known global connection result (see \cite{FP93}).
\begin{lemma}\label{lemcon}
  Let $Y$ be a locally compact metric space and let $Z$ be a compact
  subset of $Y$. Assume that any compact subset of $Y$ containing $Z$
  has nonempty boundary. Then $Y \setminus Z$ contains a connected set whose
  closure (in $Y$ ) intersects $Z$ and is not compact.
\end{lemma}

\begin{proof}[Proof of Theorem \ref{tdue}.]
Observe that since $\hat x(0)\in\big(\mathrm{pr}_1 (\mathcal{U})\big)_0$ and $\deg\big(w,\big(\mathrm{pr}_2 (\mathcal{U})\big)_0\big)$ 
is well-defined and nonzero one has that the set $Z$ is compact and nonempty. 

Since $\cl{\mathcal{N}}$ is locally compact, so is $\mathcal{N}\cap\mathcal{U}$. It is enough to prove that the 
pair
\[
 \Big(\cl{\mathcal{N}}\cap\mathcal{U},Z\Big)
\]
satisfies the assumptions of Lemma \ref{lemcon}. 

Assume by contradiction that there exists a compact set $C\subseteq\cl{\mathcal{N}}\cap\mathcal{U}$ containing $Z$ 
with empty boundary in $\cl{\mathcal{N}}\cap\mathcal{U}$. Clearly $C$ is open in $\cl{\mathcal{N}}\cap\mathcal{U}$
and, as $\mathcal{U}$ is open in $[0,\infty)\X\R^k\X M$, $C$ is open also in $\cl{\mathcal{N}}$. Thus, there exists 
an open set $A\subseteq\mathcal{U}$ with $C=\mathcal{N}\cap A$. Furthermore, $C$ being compact, one 
can assume, without loss of generality, that there exists $\varepsilon>0$ and open sets $U\subseteq\R^k$ and
$V\subseteq M$ such that, for all $\lambda\in[0,\varepsilon]$, the following equalities between slices occur:
\[
A_\lambda=A_\varepsilon=A_0,\quad\text{and}\quad A_\lambda=U\X V.
\]
In particular, since $\hat x(0)\in\big(\mathrm{pr}_1 (\mathcal{U})\big)_0$, we have $\hat x(0)\in U$. By Theorem \ref{th:formulaind} and the excision property of the degree, we get, for all $0<\lambda\leq\varepsilon$,
\begin{equation}\label{ind-deg}
 \begin{split}
   \left|\ind\big(P_T^\lambda, A_\lambda\big)\right|
     =&\left|\ind\big(P_T^\lambda, U\X V\big)\right|\\
     =&\caratt{U}\big(\hat x(0)\big)\big|\deg(w,V)\big|
     =\big|\deg(w,\big(\mathrm{pr}_2 (\mathcal{U})\big)_0)\big|
     \neq 0.
   \end{split}
\end{equation}

Now, $C$ being compact, there necessarily exists $\delta>0$ such that $C_\lambda=\emptyset$. That is,
$P_T^\delta$ is fixed point free on $A_\delta$. Thus, from the generalized homotopy invariance property of the 
degree, we have
\begin{equation}\label{zeroind}
 0=\ind\big(P_T^\delta, A_\delta\big)=\ind\big(P_T^\lambda, A_\lambda\big),
\end{equation}
for all $0<\lambda\leq\delta$. 
Clearly, \eqref{zeroind} contradicts \eqref{ind-deg} for $0<\lambda\leq\min\{\varepsilon,\delta\}$.
\end{proof}

Starting triples are important because, when some regularity is imposed on the differential equation,
they are closely related to the $T$-triples. More precisely, there exists a homeomorphism between
the relative sets that respects the notion of triviality, as shown by the following lemma:

\begin{lemma}\label{lemmatuno}
  Let $A$, $c$, $f_1$, and $f_2$ be as in Theorem \ref{tdue}, and let $X$ and $\mathcal{S}$ be, respectively, 
  the sets of $T$-triples and of starting triples for \eqref{eq:main}. Let $h\colon X\to\mathcal{S}$ be the 
  map given by $h(\lambda,x,y)=\big(\lambda,x(0),y(0)\big)$. Then $h$ is a homeomorphism that makes trivial
  $T$-triples correspond to trivial starting pairs and vice versa.
\end{lemma}

\begin{proof}
Obiously, $h$ is continuous and surjective. Since $f_1$, and $f_2$ are locally Lipschitz in $x$ and $y$, it is 
also clearly injective. The continuity with respect to initial data implies that the inverse of $h$ is 
continuous. Thus $h$ is a homeomorphism.

Finally observe that by the definition of $h$, a $T$-triple whose $\lambda$-component is $0$ can only correspond 
to a starting triple with the same property and vice versa.
\end{proof}

\begin{remark}\label{Xclos}
 Observe that by Lemmas \ref{Sclos} and \ref{lemmatuno}, given a set $C\subseteq X$ consisting of nontrivial 
 $T$-triples, it follows that a triple $(0,x,y)\in[0,\infty)\X C_T(\R^k\X M)$ belongs to the closure of $C$ if and only if $x=\hat x$ and $y=\cl{q}$ with $w(q)=0$.
\end{remark}

\smallskip
We are now ready to proceed with the proof of Theorem \ref{tuno}.

\begin{proof}[Proof of Theorem \ref{tuno}]
Assume first that $f_1$ and $f_2$ are locally Lipschitz, so that the Cauchy problem \eqref{eq:cauchy} admits 
unique solution for any $(p,q)\in\R^k\X M$ and $\lambda\geq 0$.

Consider the set
 \[
  \mathcal{S}_\Omega :=
     \big\{(\lambda,p,q)\in\mathcal{S}: \text{the solution of \eqref{eq:cauchy} is contained in $\Omega$}\big\}.
 \]
Clearly, $\mathcal{S}_\Omega$ is open in $\mathcal{S}$. Thus, there exists an open subset $\mathcal{U}_\Omega$ 
of $\mathcal{D}$ such that $\mathcal{S}\cap\mathcal{U}_\Omega=\mathcal{S}_\Omega$. Let $h$ be as in Lemma 
\ref{lemmatuno}. Let $X$ be the set of $T$-triples for \eqref{eq:main}. It is not difficult to see that $h$ maps $X\cap\Omega$ homeomorphically onto $\mathcal{S}_\Omega$.

It is not difficult to verify that $\hat x(0)\in\mathcal{O}_1(\Omega)$ if and only if $\hat x(0)\in\big(\mathrm{pr}_1 (\mathcal{U}_\Omega)\big)_0$. Also, by inspection of equation \eqref{eq:main}, one has that for any $q\in M$, $\big(0,\hat x(0),q\big) \in \mathcal{S}$ and that, conversely, the projection of $\mathcal{S}_0$ onto the third component is the whole $M$. In particular, for any $q\in M$ we have $\big(0,\hat x(0),\cl{q}\big)\in\Omega$ if and only if $\big(0,\hat x(0),q\big)\in\mathcal{S}_\Omega$. Thus, by the definition of $\mathcal{O}_2(\Omega)$ 
\[
(\mathcal{S}_\Omega)_0 = \left\{
\begin{array}{ll}
 \emptyset &\text{if $\hat x(0)\in\mathcal{O}_1(\Omega)$}\\[1mm]
 \{\hat x(0)\}\X\mathcal{O}_2(\Omega) &\text{otherwise}.
\end{array}\right.
\]
Consequently, by  the choice of $\mathcal{U}_\Omega$, when $\hat x(0)\in\mathcal{O}_1(\Omega)$ we have that $\mathcal{O}_2(\Omega)$ coincides with $\big(\mathrm{pr}_2 (\mathcal{U}_\Omega)\big)_0$, so that
\[
 \caratt{\big(\mathrm{pr}_2 (\mathcal{U}_\Omega)\big)_0}\big(\hat x(0)\big)\deg\big(w,\big(\mathrm{pr}_2 (\mathcal{U}_\Omega)\big)_0\big)
 =\caratt{\mathcal{O}_2(\Omega)}\big(\hat x(0)\big)\deg\big(w,\mathcal{O}_2(\Omega)\big) \neq 0
\]
Thus, by Theorem \ref{tdue} there exists a connected set $G$ of nontrivial starting triples for \eqref{eq:main} 
in $\mathcal{U}_\Omega$ whose closure in $\mathcal{D}$ meets the set
\begin{align*}
  \big\{(0, \hat x(0), q)\in \mathcal{U}_\Omega :   w(q) =0  \big\}
\end{align*}
and is not contained in any compact subset of $\mathcal{U}_\Omega$. It is not difficult to see that the set 
$\Gamma:=h^{-1}(G)$ has the properties required in the assertion.\smallskip

In order to conclude the proof we need to remove the local Lipschitzianity assumptions on $f_1$ and $f_2$. Denote by $\mathfrak{X}$ the subset of $X$ consisting of nontrivial $T$-triples. Let $\mathcal{Z}:=\big\{(0,\hat x,\cl{q})\in\Omega: w(q)=0\big\}$. As a consequence of Remark \ref{Xclos} we have that $\mathfrak{X}\cup\mathcal{Z}\subseteq X$ coincides with the closure $\cl{\mathfrak{X}}$ of $\mathfrak{X}$ in $[0,\infty)\X C_T(\R^k\X M)$. As in the case of starting triples, it is not difficult to show that $\cl{\mathfrak{X}}$ is locally compact. 

We only have to prove that the pair $\big(\cl{\mathfrak{X}}\cap\Omega\, , \, \mathcal{Z}\big)$ satisfies the 
assumptions of Lemma \ref{lemmatuno}. 
Assume by contradiction that there exists a relatively open compact subset $C$ of $\cl{\mathfrak{X}}\cap\Omega$ 
that contains $\mathcal{Z}$. Thus, there exists an open set $\mathcal{W}\subseteq\Omega$ such that 
$\mathcal{W}\cap\cl{\mathfrak{X}}=C$.

Since $C$ is compact, it is not difficult to show that $\mathcal{W}$ can be chosen with the following properties:
\begin{enumerate}\renewcommand{\theenumi}{\roman{enumi}}
 \item\label{l1} The closure $\cl{\mathcal{W}}$ of $\mathcal{W}$ in $[0,\infty)\X C_T(\R^k\X M)$ is a complete metric 
 space and is contained in $\Omega$;
 \item\label{l3} The boundary $\partial\mathcal{W}$ of $\mathcal{W}$ in $[0,\infty)\X C_T(\R^k\X M)$ does not 
 intersect $\cl{\mathfrak{X}}$;
 \item\label{l4} The set
 \[
  \Big\{\big(\lambda, x(t),y(t)\big)\in [0,\infty)\X\R^k\X M:(\lambda,x,y)\in\mathcal{W},\, t\in\R\Big\}
 \]
 is contained in a compact subset $K$ of $[0,\infty)\X\R^k\X M$.
\end{enumerate}
Clearly, the following subset
\[
\big\{(p,q)\in\R^k\X M: (0,\cl{p},\cl{q})\in\mathcal{W}\big\}
  =\mathcal{O}_1(\mathcal{W})\X\mathcal{O}_2(\mathcal{W})\subseteq \mathcal{O}_1(\Omega)\X\mathcal{O}_2(\Omega)
\]
is relatively compact. 

By known approximation results (see, e.g., \cite{difftop}), there exist sequences $\{f_1^i\}_{i\in\N}$ and $\{f_2^i\}_{i\in\N}$ of $T$\hbox{-}periodic smooth tangent vector fields uniformly approximating $f_1$ and $f_2$. Put 
\[
w_i(q):=\frac{1}{T}\int_0^Tf_2^i\big(t,\hat x(t),q,0\big)\, dt .
\]

As a consequence of Remark \ref{Xclos} we see that
\begin{equation}\label{condcost}
\hat x(0)\in\mathcal{O}_1(\mathcal{W}),
\end{equation} 
and also that $w_i$ is nonzero on the boundary of $\mathcal{O}_2(\mathcal{W})$ relative to $M$. Thus,
for $i\in\N$ large enough, we get
\[
\deg\big(w_i,\mathcal{O}_2(\mathcal{W})\big)
  =\deg\big(w,\mathcal{O}_2(\mathcal{W})\big)
  = \deg\big(w, \mathcal{O}_2(\Omega)\big).
\]
The last equality being a consequence of the excision property of the degree. Thus,
\begin{equation}\label{degi}
\deg\big(w_i,\mathcal{O}_2(\mathcal{W}) \big)\neq 0.
\end{equation}

Consider the system
\begin{equation}\label{duei}
\left\{
\begin{array}{l}
\dot x = A(t)x+c(t)+\lambda f_1^i(t,x,y,\lambda)\\
\dot y = \lambda f_2^i(t,x,y,\lambda),
\end{array}\right.
\end{equation}
and let $\mathfrak{X}_i$ be the set of nontrivial $T$-triples of \eqref{duei}.

For $i$ large enough, by \eqref{condcost} and \eqref{degi}, the first part of the proof applied to equation 
\eqref{duei} and the open set $\mathcal{W}$ yields a connected subset $\Gamma_i$ of $\mathfrak{X}_i\cap\mathcal{W}$ whose closure in $[0,\infty)\X C_T(\R^k\X M)$ is nonempty (as, for any $i$, it intersects  the set $\big\{(0,\hat x,\cl{q})\in\mathcal{W}:w_i(q)=0\big\}$) and is not contained in any compact subset of $\mathcal{W}$.

Let us prove that, for $i$ large enough, $\Gamma_i\cap\partial\mathcal{W}\neq\emptyset$. It is sufficient to show 
that $\cl{\mathfrak{X}_i}\cap\cl{\mathcal{W}}$ is compact. In fact, if $(\lambda,x,y)\in X_i\cap\cl{\mathcal{W}}$ 
we have, for any $t\in [0,T]$,
\[
\left|\big(\dot x(t),\dot y(t)\big)\right|^2_{k+s} \leq
\max_{\substack{(\mu,p,q)\in K\\ \tau \in [0,T]}}
\Big\{\big|A(\tau)p+c(\tau) +\mu f^i_1(\tau,p,q,\mu)|^2_k+|\mu f_1^i(\tau,p,q,\mu)|^2_s\Big\} ,
\]
where $K$ is as in \eqref{l4}, and $|\cdot|_k$, $|\cdot|_s$ and $|\cdot|_{k+s}$ denote the usual norms in $\R^k$, 
$\R^s$ and $\R^{k+s}$, respectively (recall that we are assuming $M\subseteq\R^s$). Hence, by Ascoli-Arzel\`a 
theorem, $\cl{\mathfrak{X}_i}\cap\cl{\mathcal{W}}$ is totally bounded and, consequently, compact since 
$\cl{\mathcal{W}}$ is complete by \eqref{l1}. Thus, for $i$ large enough, there exists a $T$-triple 
$(\lambda_i,x_i,y_i)\in\Gamma_i\cap\partial\mathcal{W}$ of \eqref{duei}. 
Again by Ascoli-Arzel\`a theorem, we may assume that $(x_i,y_i)\to (x_0,y_0)$ in 
$C_T(\R^k\X M)$, and $\lambda_i\to\lambda _0$ with $(\lambda _0,x_0,y_0)\in\partial\mathcal{W}$. 
Therefore
\begin{equation*}
\left\{
\begin{array}{l}
\dot x_0(t) = A(t)x(t)+c(t)+\lambda_0 f_1\big(t,x_0(t),y_0(t),\lambda_0\big)\\
\dot y_0(t) = \lambda_0 f_2\big(t,x_0(t),y_0(t),\lambda_0\big).
\end{array}\right.
\end{equation*}
Hence $(\lambda_0,x_0,y_0)$ is a $T$-triple in $\partial\mathcal{W}$. This contradicts the choice of $\mathcal{W}$, 
in particular the assumption \eqref{l3} that $\partial\mathcal{W}\cap\cl{\mathfrak{X}}=\emptyset$.
\end{proof}

\section{Further examples}\label{sec:examples}

The following example shows how Theorem \ref{tuno} can be used to get information about the set of $T$-periodic 
solutions of a parametrized differential equation containing a distributed continuous delay. The strategy
consists in reducing \eqref{eq:continuous-delay} to a coupled system of equations (see, e.g., \cite{Smi10}).
\begin{example}
Let us consider the following equation with distributed continuous delay 
\begin{equation} \label{eq:continuous-delay} 
\dot y(t) = \lambda h(t,y) + \lambda^2 \phi(t,y(t)) \int _{-\infty }^{0} y(t- \tau )e^\tau d\tau,\quad
  \lambda\geq 0,
\end{equation}
where $\phi \colon \R \X \R^n\to\R$ is continuous and $T$-periodic, $T>0$ given, in $t$. 
Since we are interested in $T$-periodic (hence bounded) solutions it makes sense to set
\begin{equation*}
  x(t)=\lambda  \int _{-\infty }^{0} y(t- \tau)e^\tau d\tau.
\end{equation*}
Differentiating this relation and integrating by parts, we get
\begin{equation*}
  \begin{aligned}
    \dot x (t) = \lambda\int_{-\infty }^{0} \dot y(t- \tau )e^\tau d\tau
    = &-\lambda\Big[ y(t- \tau )e^\tau \Big]^{\tau=0}_{\tau=-\infty} 
            + \lambda \int _{-\infty }^{0} y(t- \tau)e^\tau d\tau\\
    =& x(t) - \lambda y(t).
  \end{aligned}
\end{equation*}
Thus, given any $T$-periodic solution $y$ of \eqref{eq:continuous-delay}, we have that $(x,y)$ is an
obviously $T$-periodic solution of the coupled system
\begin{equation} \label{eq:continuous-delay-rew}
  \left\{ \begin{array}{l}
            \dot x = x - \lambda y,\\
            \dot y = \lambda \big(h(t,y)+\phi (t, y)x\big),
          \end{array}\right.
\end{equation}
which is a system of type \eqref{eq:main}. Conversely, if $(x,y)$ is a $T$-periodic solution of 
\eqref{eq:continuous-delay-rew} then $y$ is a $T$-periodic solution of \eqref{eq:continuous-delay}.
Observe that the unique $T$-periodic solution of the first equation in \eqref{eq:continuous-delay-rew},
corresponding to $\lambda=0$, is $\hat x \equiv 0$. Take $\Omega=[0,\infty)\X C_T(\R^n\X\R^n)$ and put
\[
 w(q):=\frac{1}{T}\int_0^T h \big(t, q) dt.
\]
Assuming that $\deg\big(w,\R^n\big)\ne 0$, Theorem \ref{tuno} shows that there exists an unbounded 
connected set $\Gamma$ of nontrivial $T$-triples for \eqref{eq:continuous-delay-rew}, whose closure meets 
the set
\[
 \big\{ (0,\hat x, \cl{q})\in [0,\infty)\X C_T(\R^n\X\R^n):w(\cl{q})=0\big\}.
\]
Let $\pi_2\colon [0,\infty)\X C_T(\R^n\X\R^n)\to [0,\infty)\X C_T(\R^n)$ be as in \eqref{proiez}.
Then $\pi_2(\Gamma)$ is an unbounded connected set of pairs $(\lambda,y)$, with $y$ a $T$-periodic solution 
of \eqref{eq:continuous-delay} corresponding to $\lambda$, whose closure meets 
the set
\[
 \big\{ (0, \cl{q})\in [0,\infty)\X C_T(\R^n):w(\cl{q})=0\big\}.
\]
\end{example}

\begin{example}\label{duemolle} Consider the following coupled second order ODEs:
\begin{equation}\label{ex.1}
  \left\{\begin{array}{l}
      \ddot x= -x -\delta(t) -\alpha \dot x+\mu\phi(y-x),\\
      \ddot y=\mu\phi(x-y),
    \end{array}\right.
\end{equation}
describing a mechanical system as in figure \ref{fig:duemolle}.
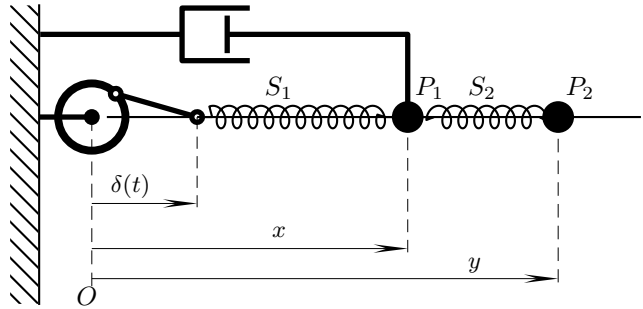
\begin{figure}[h!]
\begin{center}
  \begin{pspicture}(-4,-2.6)(5,1.8) \psset{unit=1cm}
    \psline(-2.6,0)(4.5,0)
    \psline[linewidth=2pt](-3.5,1.1)(-1.6,1.1)
    \psline[linewidth=2pt](-0.7,0.75)(-1.6,0.75)(-1.6,1.45)(-0.7,1.45)
    \psline[linewidth=2pt](1.4,0)(1.4,1.1)(-1,1.1)
    \psline[linewidth=2pt](-1,0.85)(-1,1.35)
    \pspolygon[linecolor=white,fillstyle=vlines](-3.5,-2.5)(-3.5,1.5)(-3.9,1.5)(-3.9,-2.5)
    \psline(-3.5,-2.5)(-3.5,1.5)
    \psline[linewidth=2pt]{-*}(-3.5,0)(-2.8,0)
    \pscircle[linewidth=3pt](-2.8,0){0.5}
    \psline[linewidth=2pt]{o-o}(-2.482,0.31)(-1.4,0)
    \pscoil[coilarm=.2cm,linewidth=0.8pt,coilwidth=.3cm]{-}(-1.4,0)(1.3,0)
    \qdisk(1.4,0){6pt}
    \pscoil[coilarm=.15cm,linewidth=0.8pt,coilwidth=.3cm]{-}(1.6,0)(3.3,0)
    \qdisk(3.4,0){6pt} 
    \put(-3,-2.5){$O$} 
    \psline[linewidth=0.25pt,linestyle=dashed](-2.8,0)(-2.8,-2.25)
    \put(-2.55,-1){\small $\delta(t)$} 
    \psline[arrowsize=2pt 5,arrowlength=4,linewidth=0.25pt]{->}(-2.8,-1.15)(-1.4,-1.15)
    \psline[linewidth=0.25pt,linestyle=dashed](-1.4,0)(-1.4,-1.2)
    \put(1.5,0.3){$P_1$}
    \put(-0.5,0.3){$S_1$} 
    \psline[linewidth=0.25pt,linestyle=dashed](1.4,0)(1.4,-1.8)
    \psline[arrowsize=2pt 5,arrowlength=4,linewidth=0.25pt]{->}(-2.8,-1.75)(1.4,-1.75)
    \put(-0.4,-1.6){\small $x$}
    \put(3.5,0.3){$P_2$} 
    \put(2.2,0.3){$S_2$}
    \psline[linewidth=0.25pt,linestyle=dashed](3.4,0)(3.4,-2.2) 
    \psline[arrowsize=2pt 5,arrowlength=4,linewidth=0.25pt]{->}(-2.8,-2.15)(3.4,-2.15)
    \put(2.2,-2){\small $y$}
  \end{pspicture}
\end{center}
\caption{The mechanical system of Example \ref{duemolle}}\label{fig:duemolle}
\end{figure}
There are two equal masses $P_1$ and $P_2$ confined to a linear rail, a fixed point $O$ on the rail and two connecting springs $S_1$ and $S_2$ disposed as in figure \ref{fig:duemolle}. We assume that $S_1$ is a linear spring (it obeys Hooke's law) and $S_2$ is nonlinear (we assume that the elastic force is a strictly increasing odd function $\phi$ of the displacement). Moreover, $P_1$ is subject to friction and is attached to $O$ through an actuator that displaces periodically the leftmost extremum of the spring $S_1$ by an amount $\delta(t)$.  In \eqref{ex.1}, $\alpha>0$ is the friction coefficient and $\mu\geq 0$ is a parameter used to control 
the stiffness of $S_2$.

Let us now see how Theorem \ref{tuno} can be used to get information about the set of triples $(\mu,x,y)\in (0,\infty)\X C_T^1(\R^2)$, with $(x,y)$ a solution of \eqref{ex.1} corresponding to $\mu$. Here by $C_T^1(\R^2)$ we mean the Banach space of all the $C^1$ and $T$-periodic functions $\zeta\colon\R\to\R^2$ endowed with the standard $C^1$ norm.  

Put $\mu=\lambda^2$. Equation \eqref{ex.1} can be rewritten as a first order system as follows:
\begin{equation}\label{ex.1.firstorder}
  \left\{\begin{array}{l}
      \dot x_1 = x_2,\\ 
      \dot x_2 = -x_1 -\delta(t) -\alpha x_2 +\lambda^2\phi(y_1-x_1),\\
      \dot y_1 = \lambda y_2,\\
      \dot y_2 = \lambda\phi(x_1-y_1).
    \end{array}\right.
\end{equation}
Set 
\[
\xi:=\begin{pmatrix}x_1\\ x_2\end{pmatrix},\quad \eta:=\begin{pmatrix}y_1\\ y_2\end{pmatrix},\quad A:=\begin{pmatrix}0
    & 1\\-1 &-\alpha\end{pmatrix},\quad c(t):=\begin{pmatrix}0\\-\delta(t)\end{pmatrix},
\]   
and
\[
f_1(t,\xi,\eta,\lambda):=\begin{pmatrix}0\\ \lambda\phi(y_1-x_1)\end{pmatrix},\qquad
f_2(t,\xi,\eta,\lambda):=\begin{pmatrix}y_2\\ \phi(x_1-y_1)\end{pmatrix}.
\]
Clearly, since $\alpha>0$ the eigenvalues of $A$ are not purely immaginary, hence the non-$T$-resonance condition holds.
With these definitions, equation \eqref{ex.1.firstorder} becomes
\begin{equation}\label{ex.1.block}
  \left\{\begin{array}{l}
      \dot \xi = A\xi +c(t) +\lambda f_1(t,\xi,\eta,\lambda),\\
      \dot \eta = \lambda f_2(t,\xi,\eta,\lambda),
    \end{array}\right.
\end{equation}
which has the form required by Theorem~\ref{tuno}. 

Let $\hat\xi(t)=\big(\hat x_1(t),\hat x_2(t)\big)$ be the unique $T$-periodic solution of $\dot\xi=A\xi+c(t)$.  
Take $\Omega=[0,\infty)\times C_T(\R^2\X\R^2)$, and ssume that the degree, relative to $\Omega\cap\R^2$, of the vector field
\[
w(q)=\frac{1}{T}\int_0^Tf_2\big(t,\hat\xi(t),q,0\big)\,dt 
=\begin{pmatrix} q_2\\
  \frac{1}{T}\int_0^T\phi(\hat{x}_1(t)-q_1\big)\,dt
 \end{pmatrix},
\]
where $q=(q_1,q_2)$, is nonzero. Then, by Theorem~\ref{tuno}, \eqref{ex.1.block} has an unbounded connected set of $T$-triples with $\lambda>0$ that branches from its set of trivial $T$-pairs. This corresponds to a connected set of $T$-periodic solutions of \eqref{ex.1} as claimed.
\end{example}

\section*{Aknowledgements}

The authors have been supported by the Grup\-po Na\-zio\-na\-le per l'Ana\-li\-si Ma\-te\-ma\-ti\-ca, 
la Pro\-ba\-bi\-li\-t\`a e le lo\-ro Ap\-pli\-ca\-zio\-ni (GNAMPA) of the I\-sti\-tu\-to 
Na\-zio\-na\-le di Al\-ta Ma\-te\-ma\-ti\-ca (INdAM).

\end{document}